\UseRawInputEncoding
\documentclass[12pt, reqno]{amsart}
\usepackage[margin=1in]{geometry}
\usepackage{amssymb,latexsym,amsmath,amscd,amsfonts}
\usepackage{latexsym}
\usepackage[mathscr]{eucal}
\usepackage{bm}
\usepackage{mathptmx}
\usepackage{amssymb}
\usepackage{amsthm}
\usepackage{dcolumn}
\usepackage[all]{xy}
\usepackage{enumitem}
\usepackage[utf8]{inputenc}

\def \qed {\hfill \vrule height6pt width 6pt depth 0pt}
\def\textmatrix#1&#2\\#3&#4\\{\bigl({#1 \atop #3}\ {#2 \atop #4}\bigr)}
\def\dispmatrix#1&#2\\#3&#4\\{\left({#1 \atop #3}\ {#2 \atop #4}\right)}
\newcommand{\beg}{\begin{equation}}
	\newcommand{\eeg}{\end{equation}}
\newcommand{\ben}{\begin{eqnarray*}}
	\newcommand{\een}{\end{eqnarray*}}

\newtheorem{thm}{Theorem}[section]

\newtheorem{lem}[thm]{Lemma}

\newtheorem{prop}[thm]{Proposition}
\numberwithin{equation}{section} \theoremstyle{definition}

\newtheorem{rem}[thm]{Remark}

\newtheorem{eg}[thm]{Example}

\newcommand{\HS}{\mathcal H}
\newcommand{\C}{\mathbb{C}}
\newcommand{\D}{\mathbb{D}}
\newcommand{\T}{\mathbb{T}}
\newcommand{\A}{A_r}
\newcommand{\bA}{\mathbb{Q} \mathbb{A}_r}
\newcommand{\ov}{\overline}

\def\textmatrix#1&#2\\#3&#4\\{\bigl({#1 \atop #3}\ {#2 \atop #4}\bigr)}
\def\dispmatrix#1&#2\\#3&#4\\{\left({#1 \atop #3}\ {#2 \atop #4}\right)}

\begin{document}
	\title[The $C_{1, r}$ class and quantum annulus]{Characterizations and models for the $C_{1, r}$ class and quantum annulus}
	
	\author[Pal and Tomar] {SOURAV PAL AND NITIN TOMAR}
	
	\address[Sourav Pal]{Mathematics Department, Indian Institute of Technology Bombay,
		Powai, Mumbai - 400076, India.} \email{sourav@math.iitb.ac.in}
	
	\address[Nitin Tomar]{Mathematics Department, Indian Institute of Technology Bombay, Powai, Mumbai-400076, India.} \email{tnitin@math.iitb.ac.in}		
	
	
	
	\thanks{The first named author is supported by the Seed Grant of IIT Bombay, the CDPA and the `Core Research Grant' with Award No. CRG/2023/005223 of Science and Engineering Research Board (SERB), India. The second named author is supported by the Prime Minister's Research Fellowship (PMRF ID 1300140), Government of India.}	
	
	\begin{abstract}
	For fixed $0<r<1$, let $A_r=\{z \in \mathbb{C} : r<|z|<1\}$ be the annulus with boundary $\partial \overline{A}_r=\mathbb{T} \cup r\mathbb{T}$, where $\mathbb T$ is the unit circle in the complex plane $\mathbb C$. An operator having $\ov{A}_r$ as a spectral set is called an $A_r$-\textit{contraction}. Also, a normal operator with its spectrum lying in the boundary $\partial \overline{A}_r$ is called an \textit{$A_r$-unitary}. The \textit{$C_{1,r}$ class} was introduced by Bello and Yakubovich in the following way:  
\[
C_{1, r}=\{T: T \ \mbox{is invertible and} \ \|T\|, \|rT^{-1}\| \leq 1\}.
\]
McCullough and Pascoe defined the \textit{quantum annulus} $\mathbb Q \mathbb A_r$ by
\[
\mathbb Q\mathbb A_r = \{T \,:\, T \text{ is invertible and } \, \|rT\|, \|rT^{-1}\| \leq 1  \}.
\]
If $\mathcal A_r$ denotes the set of all $A_r$-contractions, then $\mathcal A_r \subsetneq C_{1,r} \subsetneq \mathbb Q \mathbb A_r$. We first find a model for an operator in $C_{1,r}$ and also characterize the operators in $C_{1,r}$ in several different ways. We prove that the classes $C_{1,r}$ and $\mathbb Q\mathbb A_r$ are equivalent. Then, via this equivalence, we obtain analogous model and characterizations for an operator in $\mathbb Q \mathbb A_r$.
		\end{abstract}

	\maketitle

	\section{Introduction}\label{sec_intro}

	\vspace{0.2cm}

\noindent Throughout the paper, all operators are bounded linear operators acting on complex Hilbert spaces. We denote by $\D , \T , r\D, r\T$ the unit disk, the unit circle, the disk with radius $r$ and the circle with radius $r$ respectively with center at the origin in the complex plane $\C$. For a Hilbert space $\HS$, we mean by $\mathcal B(\HS)$ the algebra of operators acting on $\HS$. A \textit{contraction} is an operator whose norm is not greater than $1$. For $0<r<1$, let us consider the following annuli:
\begin{align}
\A & =\{ z \in \C \,:\, r<|z|<1 \} \ , \notag \\
\mathbb A_r &= \{ z\in \C \, : \, r< |z| < {r}^{-1} \}. \label{eqn:new-11}
\end{align}
A Hilbert space operator $T$ is said to be an $\A$-\textit{contraction} if $\overline{A}_r$ is a spectral set for $T$, which is to say that the spectrum $\sigma(T) \subseteq \overline{A}_r$ and for every rational function $f$ with poles off $\overline{A}_r$ von Neumann's inequality holds, i.e.
$
 \displaystyle \|f(T)\| \leq \sup \, |f(z)|,
$ where the supremum is taken over $\ov{A}_r$.
Here $f(T)=p(T)q(T)^{-1}$, whence $f=p/q$ with $p,q \in \C[z]$ and $q$ having no zeros in $\ov{A}_r$. An $\A$-\textit{unitary} is a normal operator having its spectrum in the boundary $\T \cup r \T$ of the annulus $\A$. In the seminal paper \cite{Agler}, Agler proved that every $\A$-contraction dilates to an $\A$-unitary which announces the success of rational dilation on an annulus. Hence, for any $\A$-contraction $T$ acting on a Hilbert space $\HS$, there is a Hilbert space $\mathcal K \supseteq \HS$ and an $\A$-unitary $N$ such that $f(T)=P_{\HS}f(N)|_{\HS}$ for every rational function $f$ with poles off $\ov{A}_r$. This path-breaking work due to Agler motivates numerous mathematicians to study further the functions and operators associated with an annulus which leads to exciting works like \cite{McColloughII, comm lifting, McColloughIII, McColloughIV, DritschelII, DritschelI}. Also, see the references therein. Recently, Bello and Yakubovich \cite{Dmitry} introduced two important classes of operators associated with an annulus, namely the $C_{\alpha}$ and $C_{1,r}$ classes which were defined in the following way:
\begin{align*}
& C_{\alpha} = \{T \, :\, T \text{ is an invertible operator and } \, \alpha(T)=-{T^*}^2T^2+(1+r^2)T^*T-r^2I \, \geq 0  \} \, , \\
& C_{1,r} = \{T \,:\, T \text{ is an invertible operator and } \, \|T\|, \, \|rT^{-1}\| \leq 1  \}.
\end{align*}
We mention two important facts about these classes: first, the operators in $C_{\alpha}, \, C_{1,r}$ classes have their spectrums in $\ov{A}_r$ and second, if $\mathcal A_r$ denotes the set of all $\A$-contractions, then $\mathcal A_r, \, C_{\alpha}, \, C_{1,r}$ form a strictly increasing chain as was proved in \cite{Dmitry}.
\begin{thm}[\cite{Dmitry}, Theorem 1.1]\label{thm_3class}
$\mathcal A_r \subsetneq C_\alpha \subsetneq C_{1, r}$ .
\end{thm}
Also, an explicit model was constructed for an operator in the $C_{\alpha}$ class by Bello and Yakubovich, see Theorem 1.2 in \cite{Dmitry}. Interestingly, McCullough and Pascoe \cite{Pas-McCull} considered the annulus $\mathbb A_r$ as in (\ref{eqn:new-11}) and introduced the \textit{quantum annulus} $\mathbb Q \mathbb A_r$ which consists of invertible operators $T$ such that both $rT$ and $rT^{-1}$ are contractions, i.e.
\[
\mathbb Q \mathbb A_r = \{T\,:\, T \text{ is an invertible operator and } \, \|T\|, \, \|T^{-1}\| \leq r^{-1}  \}.
\]
Evidently, $C_{1,r} \subsetneq \mathbb Q \mathbb A_r$ which stretches the increasing chain of Theorem \ref{thm_3class} one more step. In \cite{Pas-McCull}, McCullough and Pascoe found the following model theorem for an operator in $\mathbb Q \mathbb A_r$.

\begin{thm}[\cite{Pas-McCull}, Theorem 1.1]\label{thm_model_QAr}
An invertible operator $T$ acting on a Hilbert space $\mathcal{H}$ is in $\mathbb{Q} \mathbb{A}_r$ if and only if there exists an invertible operator $J$ acting on a larger Hilbert space $\mathcal{K}$ containing $\mathcal{H}$ such that 
$T^n=P_\mathcal{H}J^n|_\mathcal{H}$ for all $n \in \mathbb{Z}$ and $J$, upto unitary equivalence, takes the following form
\[
J=U\begin{bmatrix}
rI_{\mathcal{K}_{0}} & 0\\
0 & r^{-1}I_{\mathcal{K}_{1}}
\end{bmatrix}
\]
where $U$ is a unitary, $\mathcal{K}=\mathcal{K}_0 \oplus \mathcal{K}_1$ and $P_\HS$ denotes the orthogonal projection of $\mathcal K$ onto $\HS$.
\end{thm}
These wider classes of operators generalizing the $\A$-contractions have been extensively studied in recent past in \cite{Mittal, Tsikalas, TsikalasII} also.

\medskip

In this article, we further analyse the $C_{1,r}$ class and the quantum annulus. We find the following model-cum-characterization theorem for the $C_{1,r}$ class in terms of a pair of $\A$-unitaries. Before we state the theorem, let us mention that any rational function $f$ with poles off $\ov{A}_r$ can be represented as
\begin{equation} \label{eq:1e}
   f(z)=\frac{p(z)}{q_1(z)q_2(z)}, 
\end{equation}
where $p, q_1$ and $q_2$ are polynomials in $\C[z]$ such that the zeros of $q_1$ and $q_2$ lie in $\mathbb{C}\setminus\overline{\mathbb{D}}$ and $r\mathbb{D}$ respectively. 

\begin{thm}\label{Main}

 	Let $T$ be an invertible operator acting on a Hilbert space $\mathcal{H}$. Then $T \in C_{1, r}$ if and only if there is a Hilbert space $\mathcal{K} \supseteq \mathcal H$, an $A_r$-unitary $N$ on $\mathcal K$ and a self adjoint unitary $F $ on $\mathcal K$ such that 
 \[
 f(T)=P_\mathcal H \bigg(p(N)q_1(N)^{-1}q_2(FNF)^{-1}\bigg) \bigg|_\mathcal H
 \]
 for every rational function $f$ with poles off $\overline{A}_r$ $(\mbox{as in} \ (\ref{eq:1e}))$.

 \end{thm}

Note that both $N$ and $FNF^{-1}$ as in the model above are $\A$-unitaries. Next, we obtain the following characterizations for an operator in the $C_{1,r}$ class. We mention that the equivalence of the conditions $(1)$ and $(2)$ of this theorem follows from Proposition 3.2 in \cite{Pas-McCull}. However, we present a different proof for this part in this paper.
 
\begin{thm}\label{thm_C1r_pencil}
Let $T$ be an invertible operator acting on a Hilbert space $\HS$. Then the following are equivalent:
\begin{enumerate}
\item $T \in C_{1, r}$ ;
\item $(1+r^2)I_\HS-T^*T-r^2T^{-1}(T^{-1})^* \geq 0$ ;
\item $(T^*T)^{1\slash 2} \in C_\alpha$ ;
\item $(T^*T)^{1\slash 2}$ is an $A_r$-contraction ;
\item there exist a unitary $U$ and an $\A$-contraction $P$ on $\HS$ such that $T=UP$.
\end{enumerate}
\end{thm}

These results will be proved in Section \ref{sec_C1r}. Though, we have $C_{1,r} \subsetneq \bA$ by definition, actually these two classes are comparable. To see this, let us consider the map
$
\varphi: \mathbb{A}_r \to A_{r^2}$ defined by $ \varphi(z)=rz$,
which is a biholomorphism with $\varphi^{-1}(z)=r^{-1}z$. Now one can easily prove the following lemma that establishes the equivalence of $C_{1,r}$ and $\bA$.

\begin{lem}\label{lem_equiv}
An operator $T \in C_{1, r}$ if and only if $r^{-1\slash 2} T \in \mathbb{Q} \mathbb{A}_{\sqrt{r}} \ $. Also, $T \in \mathbb{Q} \mathbb{A}_r$ if and only if $rT \in C_{1, r^2}$.
\end{lem}
Hence, any result that holds for the $C_{1,r}$ class must have an analogue for the quantum annulus. So, we have the following model theorem and characterizations for $\bA$ that are analogous to Theorems \ref{Main} \& \ref{thm_C1r_pencil} respectively.
\begin{thm}\label{thm_404}
Let $T$ be an invertible operator acting on a Hilbert space $\mathcal{H}$. Then $T \in \mathbb{Q} \mathbb{A}_r$ if and only if there is a Hilbert space $\mathcal{K} \supseteq \mathcal H$, a normal operator $N$ on $\mathcal K$ with $\sigma(N) \subseteq  \partial \overline{\mathbb{A}}_r$ and a self adjoint unitary $F $ on $\mathcal K$ such that 
 \[
 g(T)=P_\mathcal H \bigg(p(N)q_1(N)^{-1}q_2(FNF)^{-1}\bigg) \bigg|_\mathcal H
 \]
 for every rational function $g=p\slash q_1q_2$ with the zeros of $q_1$ and $q_2$ in $\C \setminus r^{-1} \overline{\D}$ and $r\D$ respectively.
\end{thm} 

\begin{thm}\label{thm_QAr_pencil}
Let $T$ be an invertible operator on a Hilbert space $\HS$. Then the following are equivalent:
\begin{enumerate}
\item $T \in \mathbb{Q} \mathbb{A}_{r}$;
\item $(r^{-2}+r^2)I_\HS-T^*T-T^{-1}(T^{-1})^* \geq 0$;
\item $\overline{\mathbb{A}}_r$ is a spectral set for $(T^*T)^{1\slash 2}$;
\item there exist an operator $P$ with $\overline{\mathbb{A}}_r$ as a spectral set and a unitary $U$on $\HS$ such that $T=UP$.
\end{enumerate}
\end{thm} 

Note that the model that we obtain for $\mathbb Q \mathbb A_r$ as in Theorem \ref{thm_404} consists of a pair of normal operators having their spectrums on the boundary of the annulus $\bA$. Since Lemma \ref{lem_equiv} allows us to move back and forth between $C_{1,r}$ and $\mathbb{Q} \mathbb{A}_r$, we do not want to miss the opportunity to achieve an alternative model for the $C_{1,r}$ class which goes parallel with the model for $\bA$ obtained by McCullough and Pascoe (i.e. Theorem \ref{thm_model_QAr}).

\begin{thm}\label{thm_C1r_MP}
An invertible operator $T$  acting on a Hilbert space $\mathcal{H}$ is in $C_{1, r}$ if and only if there is an invertible operator $J$ on a Hilbert space $\mathcal{K} \supseteq \mathcal{H}$ such that $T^n=P_\mathcal{H}J^n|_\mathcal{H}$ for every $n \in \mathbb{Z}$ and  $J$ admits the following form:
\[
J=U\begin{bmatrix}
rI_{\mathcal{K}_0} & 0\\
0 & I_{\mathcal{K}_1}
\end{bmatrix},
\] 
where $U$ is a unitary and $\mathcal{K}=\mathcal{K}_0 \oplus \mathcal{K}_1$.
\end{thm}
 
We prove the results associated with $\bA$ in Section \ref{Sec_quant}. In Section \ref{sec_prep}, we prove a few relevant and preparatory results.

\vspace{0.2cm}

\section{Preparatory results}\label{sec_prep}

\vspace{0.2cm}

\noindent We begin with a famous result due to Ando which states that a pair of commuting contractions $T_1, T_2$ can always be lifted simultaneously to a pair of commuting unitaries $U_1,U_2$.

\begin{thm}[Ando, \cite{Ando}]\label{Ando}

		Given a commuting pair of contractions $(T_1, T_2)$ on a Hilbert space $\mathcal{H},$ there exists a commuting pair of unitaries $(U_1, U_2)$ on a Hilbert space $\mathcal{K}$ containing $\mathcal{H}$ such that 
		\[
		p(T_1, T_2)=P_\mathcal{H}p(U_1, U_2)|_\mathcal{H}
		\]
		for every polynomial $p$ in two variables.

	\end{thm}

We now state and prove a few basic properties of an operator in $C_{1,r}$ class which will be used in the proof of the main theorems.

\begin{lem}\label{||T||_in_Ar}

For every operator $T$ in $C_{1,r}$ class, we have $r \leq  \|T\| \leq 1.$

\end{lem}

\begin{proof}
Since $T$ is in $C_{1,r}$ class, we have that $\sigma(T) \subseteq \ov{A}_r$ and $\|T\|, \|rT^{-1}\| \leq 1$. Let if possible, $\|T\| <r.$ For any $\lambda=se^{i \theta}$ for $s \geq r$ and $\theta \in [0,2\pi],$ we have that $\|T\| <r \leq s= |\lambda|$ and hence, $(I-\lambda^{-1}T)$ is invertible which yields that $\lambda \notin \sigma(T).$ Consequently, we get that $\sigma(T) \subseteq r\mathbb{D}$ which is a contradiction. Hence, $r \leq \|T\| \leq 1.$
\end{proof}

 The converse of Lemma \ref{||T||_in_Ar} does not hold. Indeed, if we choose $r=1/2$ and
 $
 T =\begin{bmatrix}
 0 & 1 \\
 \frac{1}{100} & 0 
 \end{bmatrix},
 $
 then $T$ is invertible and $r \leq \|T\| \leq 1$. However, $T$ does not belong to $C_{1,r}$ as the spectrum $\sigma(T)$ is not contained in $\ov{A}_r$. Even more is true, the converse to Lemma \ref{||T||_in_Ar} does not hold for $A_r$-contractions, i.e. an invertible operator $T$ with $\sigma(T) \subseteq \overline{A}_r$ and $r \leq \|T\| \leq 1$ is not necessarily an $A_r$-contraction. The following example due to G. Misra \cite{Misra} shows this clearly. Before going to the example let us state an interesting result due to Williams \cite{Williams} showing an interplay between spectral set and complete non-normality. It is to mention that an operator $T \in \mathcal{B}(\mathcal{H})$ is said to be \textit{completely non-normal} if there is no nonzero closed subspace $\HS_1$ of $\HS$ that reduces $T$ and $T|_{\HS_1}$ is normal. 

\begin{thm}[Williams, \cite{Williams}] \label{Williams}

	If an operator $T$ on a finite dimensional space is completely non-normal and $\|T\|=1$, then $\overline{\mathbb{D}}$ is a minimal spectral set for $T$, i.e. no proper closed subset of $\overline{\mathbb{D}}$ is a spectral set for $T$.   

\end{thm}

\begin{eg}\label{Example}
	For $0<r<1,$ the matrix 
	$
	T=\begin{bmatrix}
		\sqrt{r} &  1-r\\
		0 & \sqrt{r} 
	\end{bmatrix} 
	$
	is invertible and $\sigma(T)=\{\sqrt{r}\} \subseteq A_r.$ Note that 
	\[
	T^*T=\begin{bmatrix}
		r & \sqrt{r}(1-r)\\
		\sqrt{r}(1-r) & r+(1-r)^2\\ 
	\end{bmatrix} 
	\]
	and thus it follows that  $\sigma(T^*T)=\{1, r^2\}.$ Therefore, 
	$
	\|T\|^2=\|T^*T\|=1.
	$
	Hence, $r\leq \|T\| \leq 1$. It is not difficult to see that $T \in C_{1, r}$. Since $\|T\|=1$ and $T$ is completely non-normal, Theorem \ref{Williams} implies that $\overline{\mathbb{D}}$ is a minimal spectral set for $T.$ Hence $\overline{A}_r$ cannot be a spectral set for $T$. \qed
\end{eg}

\begin{lem}\label{T and rT^-1}
Let $T$ be an invertible operator acting on a Hilbert space $\HS$. Then

\begin{itemize}

\item[(a)] $T$ is an $A_r$-contraction if and only if $rT^{-1}$ is an $A_r$-contraction ;

\item[(b)] $T$ is in $C_{1,r}$ if and only if $rT^{-1}$ is in $C_{1,r}$ .

\end{itemize}	

\end{lem}

\begin{proof}

The part-(b) is obvious and so we prove only part-(a).
	Let $T$ be an $A_r$-contraction. Then by Lemma \ref{||T||_in_Ar}, we have $\sigma(T)\subseteq \overline{A}_r$ and $r \leq \|T\| \leq 1$. So, it follows from the Spectral Mapping Theorem that	
\[
\sigma(rT^{-1})=\{ r\slash \lambda \ :\  \lambda \in \sigma(T) \} \subseteq \overline{A}_r.
\]
Let $f$ be any rational function with poles off $\overline{A}_r.$ Then we define a rational function $g(z)=f(rz^{-1})$, where $z \mapsto rz^{-1}$ is an automorphism of the annulus $A_r$. Evidently, $g$ has its poles off $\overline{A}_r.$ Now, 
	\begin{equation*}
		\begin{split}
			\|f(rT^{-1})\| =\|g(T)\|
			& \leq \sup \{|g(z)| : r \leq |z| \leq 1\} \\
			& = |g(w)| \ \ \ \     \mbox{for some $w \in \mathbb{T} \cup r\mathbb{T}$\ \ \ \ (by the Maximum-modulus principle)}\\ 
			&=|f(rw^{-1})| \\
			& \leq \sup \{|f(z)| : r \leq |z| \leq 1\}. \\
		\end{split}
	\end{equation*}
Therefore, $\overline{A}_r$ is a spectral set for $rT^{-1}$. Again if $\overline{A}_r$ is a spectral set for $S=rT^{-1}$, then by previous part of the proof we have that $rS^{-1}=T$ is also an $A_r$-contraction and the proof is complete.
\end{proof}

Let $T$ be an operator on a Hilbert space $\mathcal H$ and let $\gamma$ be a simple closed curve in $\C$ such that $\sigma(T)$ is contained in the interior of $\gamma$. If $f$ is a holomorphic function on and in the interior of $\gamma$, then $f(T)$ can be defined in the following way:
\begin{equation} \label{eqn:new-021}
	f(T):=\frac{1}{2\pi i}\underset{\gamma}{\int}f(w)(w-T)^{-1}dw.
\end{equation}
It is merely mentioned that the above integral and hence the definition of $f(T)$ is independent of the choice of $\gamma$. Before going to the next proposition we state a classic result whose proof is a routine exercise.

\begin{lem}\label{continuity}
	Let $T \in \mathcal{B}(\mathcal{H})$ and let $\Omega$ be an open set containing $\sigma(T).$ If a sequence $\{f_n\}$ of holomorphic functions on $\Omega$ converges uniformly to a function $f$ on every compact subset of $\Omega$, then $f$ is holomorphic on $\Omega$ and $\{f_n(T)\}$ as in $($\ref{eqn:new-021}$)$ converges to $f(T)$ in operator norm.
\end{lem} 

Every rational function $f$ with poles off $\overline{A}_r$ is analytic in an open neighbourhood containing $\overline{A}_r$ and thus has a unique Laurent series $f(z)=\overset{\infty}{\underset{j=-\infty}{\sum}}f_jz^j$. We show that for $T \in C_{1, r}$ acting on $\mathcal{H},$ the series $\overset{\infty}{\underset{j=-\infty}{\sum}}f_jT^j$ defines an operator on $\mathcal{H}$ and $f(T)=p(T)q(T)^{-1}=\overset{\infty}{\underset{j=-\infty}{\sum}}f_jT^j$, where $f=p/q$ with $q$ having no zeros inside $\overline{A}_r$. 

\begin{prop}\label{Laurent}

	 Given an operator $T \in C_{1, r}$ on a Hilbert space $\mathcal{H}$ and a rational function $f$ with poles off $\overline{A}_r,$ the series $\overset{\infty}{\underset{j=-\infty}{\sum}}f_jT^j$ defines a bounded linear operator on $\mathcal{H}$ and is same as $f(T)$.

\end{prop}

\begin{proof}
Since $T \in C_{1, r}$, it follows from the Spectral Mapping Theorem that $\sigma(T) \subseteq \overline{A}_r$. Thus $f(T)$ is well-defined. The sequence $\tilde{f}_n(z)=\overset{n}{\underset{j=-n}{\sum}}f_jz^j$ converges uniformly to $f(z)$ on $\overline{A}_r$ and we have by Lemma \ref{continuity} that $\tilde{f}_n(T)$ converges to $f(T)$ in operator norm topology. Thus, $f(T)=\overset{\infty}{\underset{j=-\infty}{\sum}}f_jT^j.$
\end{proof}

In order to prove that an $A_r$-contraction $T \in \mathcal B(\mathcal H)$ admits an $A_r$-unitary dilation $N$ on $\mathcal K \supseteq \mathcal H$, one needs to show that $f(T)=P_\mathcal{H}f(N)|_\mathcal H$ for every rational function $f$ with poles off $\overline{A}_r$. The next lemma shows that instead of all rational functions $f$ it suffices to consider only the integral powers of $z$, i.e. the functions of the type $z^j$ for $j \in \mathbb Z$.

\begin{lem}\label{lem_int_power}

	For operators $T$ and $N$ in $C_{1, r}$ acting on Hilbert spaces $\HS$ and $\mathcal{K}$ respectively with $\mathcal{K}\supseteq \mathcal{H}$, the following are equivalent:
	\begin{enumerate}
		\item $f(T)=P_\mathcal{H}f(N)|_\mathcal H$ \ for every rational function $f$ with poles off $\overline{A}_r$ ;
		
		\item $T^j=P_\mathcal{H}N^j|_\mathcal H$ \ for every $j \in \mathbb{Z}.$
	\end{enumerate}

\end{lem}

\begin{proof}
$(1) \implies (2)$ is obvious. We prove $(2) \implies (1)$. Let $f$ be a rational function with poles off $\overline{A}_r$ and let $f(z)=\overset{\infty}{\underset{j=-\infty}{\sum}}f_jz^j$ be its Laurent series. We have by Proposition $\ref{Laurent}$ that $f(T)=\overset{\infty}{\underset{j=-\infty}{\sum}}f_jT^j$ and $f(N)=\overset{\infty}{\underset{j=-\infty}{\sum}}f_jN^j$. For any $h\in \mathcal H$, it follows that
\[
	f(T)h
=\overset{\infty}{\underset{j=-\infty}{\sum}}f_jT^jh
=\overset{\infty}{\underset{j=-\infty}{\sum}}f_j(P_\mathcal{H}N^jh)
=P_\mathcal{H}\overset{\infty}{\underset{j=-\infty}{\sum}}f_jN^jh
=P_\mathcal{H}f(N)h.
\]
\end{proof}

Next, we show that the three classes $\mathcal A_r, C_\alpha, C_{1,r}$ as in Theorem \ref{thm_3class} agree under subnormality condition. To do so, it suffices to show that every subnormal $C_{1, r}$ operator is an $\A$-contraction. First we state an elementary result whose proof is a routine exercise.

\begin{lem}\label{lem_normal_Ar}
Let $N$ be a normal operator with $\sigma(N) \subseteq \overline{A}_r$. Then $N$ is an $\A$-contraction.
\end{lem}

The next two results are also important in the context of this article.

\begin{prop} \label{prop:new-021}
A subnormal operator $T \in C_{1, r}$ if and only if $T$ is an $\A$-contraction.
\end{prop}

\begin{proof}
Let $T \in C_{1, r}$ be subnormal. We have that $\sigma(T) \subseteq \overline{A}_r$. Let $N$ be the minimal normal extension of $T$. It follows from Theorem 2.11 in Chapter 2 of \cite{Conway} that $\sigma(N) \subseteq \sigma(T)$  and so, $\sigma(N) \subseteq \overline{A}_r$. Note that $N$ is invertible and since $T=N|_\HS$, we have that $T^{-1}=N^{-1}|_\HS$. Therefore, $T^m=N^m|_\HS$ for every $m \in \mathbb Z$. It follows from Lemma \ref{lem_int_power} that $f(T)=f(N)|_\HS$ for every rational function $f$ with poles outside $\overline{A}_r$. 
 By Lemma \ref{lem_normal_Ar}, $N$ is an $\A$-contraction and so, $\|f(T)\| \leq \|f(N)\| \leq \sup \{|f(z)| : z \in \overline{A}_r\}$ for every rational function $f$ with poles off $\overline{A}_r$. Thus, $T$ is an $\A$-contraction. The converse follows from Theorem \ref{thm_3class}.
\end{proof}

\begin{prop} \label{prop:new-022}
A subnormal operator $T \in \mathbb{Q} \mathbb{A}_r$ if and only if $\overline{\mathbb{A}}_r$ is a spectral set for $T$.
\end{prop}

\begin{proof}
Let $T \in \mathbb{Q}\mathbb{A}_r$ be subnormal. We have by Lemma \ref{lem_equiv} that $rT \in C_{1, r^2}$ is also subnormal. It follows from Proposition \ref{prop:new-021} that $rT$ is an $A_{r^2}$-contraction. Since $\varphi: \mathbb{A}_r \to A_{r^2}, \varphi(z)=rz$ is a biholomorphism, we have that $\varphi^{-1}(rT)=T$ has $\varphi^{-1}(\overline{A}_{r^2})=\overline{\mathbb{A}}_r$ as a spectral set. The converse is trivial.
\end{proof}

However, an operator in $\mathbb Q \mathbb A_r$ may not always have $\ov{\mathbb A}_r$ as a spectral set. Actually, the class of operators having $\ov{\mathbb A}_r$ as a spectral set is contained in $\mathbb Q \mathbb A_r$ and it follows trivially from the von Neumann's inequality. The following example shows that the containment is strict.

\begin{eg}\label{ExampleII}
	For $0<r<1$, consider the matrix  
	$
	T=\begin{bmatrix}
		1 &  -r+r^{-1}\\
		0 & 1 
	\end{bmatrix} 
	$.
It is not difficult to see that $T$ is invertible and $\|T\|=\|T^{-1}\| =r^{-1}$. Therefore, $T \in \mathbb{Q} \mathbb{A}_r$. Since $\|rT\|=1$ and $rT$ is completely non-normal, Theorem \ref{Williams} implies that $\overline{\mathbb{D}}$ is a minimal spectral set for $rT$. Consequently, $r^{-1}\overline{\mathbb{D}}$ is a minimal spectral set for $T$. Hence, $\overline{\mathbb A}_r$ cannot be a spectral set for $T$. \qed
\end{eg}

\begin{lem}\label{II_Laurent}

 Let $T \in C_{1, r}$. If $f$ as in $($\ref{eq:1e}$)$ is any rational function with poles off $\overline{A}_r$, then 
    \[
    f(z)=\bigg(\overset{\infty}{\underset{n=0}{\sum}}a_{n}z^n\bigg) \bigg(\overset{\infty}{\underset{n=0}{\sum}}b_{n}z^{-n}\bigg) \; \text{ and } \; f(T)=\bigg(\overset{\infty}{\underset{n=0}{\sum}}a_{n}T^n\bigg) \bigg(\overset{\infty}{\underset{n=0}{\sum}}b_{n}T^{-n}\bigg), 
    \]
    for some scalar coefficients $a_n, b_n$.

\end{lem}

\begin{proof}
For any rational function $f$ with poles off $\overline{A}_r$ we have from (\ref{eq:1e}) that
\[
f(z)=\frac{p(z)}{q_1(z)q_2(z)},
\]
where $p, q_1$ and $q_2$ are polynomials with zeros of $q_1$ in $\mathbb{C}\setminus\overline{\mathbb{D}}$ and zeros of $q_2$ in $r\mathbb{D}.$ Suppose $\alpha_1, \dots , \alpha_k \in \mathbb C \setminus \overline{\mathbb D}$ are the zeros of $q_1$ and $\beta_1 ,\dots , \beta_l \in r \mathbb D$ are the zeros of $q_2$. Then
\[ 
q_1(z)  =\alpha (z-\alpha_1)(z-\alpha_2)\dotsc (z-\alpha_k) \quad \text{and} \quad
q_2(z)  =\beta(z-\beta_1)(z-\beta_2)\dotsc (z-\beta_l),
\]
for some $\alpha, \beta \in \mathbb C$. Now for each $\alpha_j,\beta_i$ and for any $z \in \overline{A}_r$, we have 
\[
\frac{1}{z-\alpha_j}
= \frac{-1}{\alpha_j(1-z\slash \alpha_j)}
=-\overset{\infty}{\underset{n=0}{\sum}}\frac{z^n}{\alpha_j^{n+1}} 
\quad \text{and} 
\quad
\frac{1}{z-\beta_i}
=\frac{1}{z(1-\beta_i \slash z)}
=\overset{\infty}{\underset{n=0}{\sum}}\beta_i^nz^{-(n+1)},
\]
where both the series converge uniformly on $\overline{A}_r.$ Thus, for any $z\in \overline{A}_r$ we have that 
\[
\frac{1}{q_1(z)}=\overset{\infty}{\underset{n=0}{\sum}}q_{n1}z^n \quad \mbox{and} \quad \frac{1}{q_2(z)}=\overset{\infty}{\underset{n=0}{\sum}}\frac{q_{n2}}{z^n}
\]
for some scalar coefficients $q_{n1}, q_{n2}$. Consequently,
\[
f(z)=p(z)\bigg(\overset{\infty}{\underset{n=0}{\sum}}q_{n1}z^n\bigg) \bigg(\overset{\infty}{\underset{n=0}{\sum}}q_{n2}z^{-n}\bigg).
\]
Evidently it follows from Lemma \ref{continuity} that
\[
f(T)=p(T)\bigg(\overset{\infty}{\underset{n=0}{\sum}}q_{n1}T^n\bigg) \bigg(\overset{\infty}{\underset{n=0}{\sum}}q_{n2}T^{-n}\bigg)
\]
and the proof is complete.
\end{proof}

We conclude this Section by recalling from the literature a useful result on joint spectrum.

\begin{thm}[\cite{Taylor}, Theorem 4.9]\label{Taylor}
	If $\underline{T}=(T_1, \dotsc, T_n)$ is a tuple of commuting operators on a Hilbert space $X$ and if $\sigma_T(\underline{T})=K_1 \cup K_2$, where $K_1$ and $K_2$ are disjoint compact sets in $\mathbb{C}^n,$ then there are closed linear subspaces $X_1$ and $X_2$ of $X$ such that
	\begin{enumerate}
		\item $X=X_1\oplus X_2;$
		\item $X_1, X_2$ are invariant under any operator which commutes with each $T_k;$
		\item $\sigma_T(\underline{T}|_{X_1})=K_1$ and $\sigma_T(\underline{T}|_{X_2})=K_2$, where $\underline{T}|_{X_i}=(T_1|_{X_i}, \dotsc, T_n|_{X_i})$ for $i=1,2.$ 
	\end{enumerate}
\end{thm}

\vspace{0.2cm}

\section{Characterizations and operator model for the $C_{1, r}$ class}\label{sec_C1r}

\vspace{0.2cm}

\noindent While investigating the success or failure of rational dilation on the closure of a domain $\Omega$, a primary step towards the endeavour is to study the normal operators having their spectrum in the boundary $\partial \ov{\Omega}$. Such operators constitute an analogue of unitaries, which are normal operators associated with the boundary of the unit disk. For an annulus $\A$, they are $\A$-unitaries. In this Section, we first characterize an $\A$-unitary as a direct sum $U_1 \oplus rU_2$ for a pair of unitaries $U_1,U_2$. Using this characterization, we frame a model (see Theorem \ref{Main}) for an operator in $C_{1, r}$ class and the model consists of a pair of $\A$-unitaries. For the sake of brevity,	we fix the following notations for a contraction $T \in \mathcal{B}(\HS)$:
	\[
		T(n)=T^n \ \  (n \geq 1), \quad T(0)=I_\HS, \quad  T(n)=T^{*|n|} \ \ (n \leq -1). 
	\]
Evidently, $\|T\| \leq 1$ if and only if $I-T^*T \geq 0$. Let $D_T=(I-T^*T)^{1\slash 2}$ be the unique positive square root of a contraction $T$.

\begin{thm}\label{Ar_unitary}

	An operator $T \in \mathcal{B}(\mathcal{H})$ is an $A_r$-unitary if and only if $\mathcal{H}$ decomposes into an orthogonal sum $\mathcal{H}=\mathcal{H}_1\oplus \mathcal{H}_2$ such that $\mathcal{H}_1, \mathcal{H}_2$ reduce $T$ and $T_1= T|_{\mathcal{H}_1}, T_2=rT^{-1}|_{\mathcal{H}_2}$ are unitaries. This decomposition is uniquely determined. Indeed, we have that
	\begin{equation*}
		\mathcal{H}_1=\{h \in \mathcal{H} \ : \ \|T^nh\|=\|h\|=\|T^{*n}h\|, \;\;  n=1,2,\dotsc \}
	\end{equation*}
	and 
	\begin{equation*}
		\mathcal{H}_2=\{h \in \mathcal{H} \ : \ \|T^nh\|=r^n\|h\|=\|T^{*n}h\|, \;\; \ n=-1,-2,\dotsc \}.
	\end{equation*}
	The space $\mathcal{H}_1$ or $\mathcal{H}_2$ may coincide with the trivial space $\{0\}$. With respect to the decomposition $\mathcal{H}=\mathcal{H}_1\oplus \mathcal{H}_2$, $T$ has the following block-matrix form: 
	\[
	T=	\begin{bmatrix}
		T_1 &  0     \\
		0   &  rT_2^{-1}  \\ 		
	\end{bmatrix}.
	\]
	
\end{thm}	

\begin{proof}
	Since $T$ is an $A_r$-unitary, we have by Theorem \ref{thm_3class} that $T$ is in $C_{1,r}$ and thus $T$ and $rT^{-1}$ are contractions. Hence, $T(n)$ and $(rT^{-1})(n)$ are contractions for every integer $n$. For each fixed $n \in \mathbb{Z}$, it is evident that 
	\[
	\text{Ker}\,D_{T(n)} =\{h\in \mathcal H \,:\, \|T(n)h\|=\|h\|  \} \quad \text{and} \quad \text{Ker}\,D_{(rT^{-1})(n)} =\{h\in \mathcal H \,:\, \|rT^{-1}(n)h\|=\|h\|  \}.
	\]
	Therefore, we have
	\[
		\mathcal{H}_1=\underset{n \in \mathbb{Z} \setminus \{0\}}{\bigcap}\mbox{Ker}\ D_{T(n)} \quad \text{and} \quad\mathcal{H}_2=\underset{n \in \mathbb{Z} \setminus \{0\}}{\bigcap}\mbox{Ker}\ D_{(rT^{-1})(n)}.
		\]
	It is obvious that $\mathcal{H}_1$ and $\mathcal{H}_2$ are closed linear subspaces of $\mathcal{H}$. For any $h \in \mathcal{H}_1$, we have 
	\begin{align*}
		& \|T^nTh\|=\|T^{n+1}h\|=\|h\|=\|Th\| \ \ (n=0,1,2, \dotsc), \\
		\|T^{*n}Th\| &=\|T^{*n-1}T^*Th\|=\|T^{*n-1}h\|=\|h\|=\|Th\| \ \ (n=1,2, \dotsc),
	\end{align*}
	which follow from the fact that for a contraction $T,$ $\|Th\|=\|h\|$ if and only if $T^*Th=h$. Hence, $Th \in \mathcal{H}_1$. Similarly, one can show that $T^*h \in \mathcal{H}_1$. Thus $\mathcal{H}_1$ reduces $T$. A similar argument implies that $\mathcal{H}_2$ reduces $T$. If we set $T_1=T|_{\mathcal{H}_1}$ and $T_2=rT^{-1}|_{\mathcal{H}_2},$ then it follows from the definition of $\HS_1$ and $\HS_2$ that $T_1$ and $T_2$ are unitaries on $\HS_1$ and $\HS_2$ respectively. Consequently, we have that 
	\begin{equation*}
		\langle h_1, h_2\rangle=\langle T^{*}Th_1, (rT^{-1})(rT^{-1})^*h_2\rangle=\langle T^{-1}T^{-1*}T^*Th_1, r^2h_2 \rangle=r^2 \langle h_1,h_2\rangle  
	\end{equation*}
for $h_1 \in \HS_1, h_2 \in \HS_2$ and so, $\langle h_1, h_2 \rangle =0$ as $0<r<1$. Hence, $\mathcal{H}_1$ and $\mathcal{H}_2$ are orthogonal. Consider the subspace $\mathcal{H}_3=\mathcal{H} \ominus \big(\mathcal{H}_1 \oplus \mathcal{H}_2\big)$ which reduces $T$ and thus, $T_3=T|_{\mathcal{H}_3}$ is normal. Since $\sigma(T_3) \subseteq \sigma(T) \subseteq \mathbb{T} \cup r\mathbb{T}$, we have that $\sigma(T_3)=K_1 \cup K_2,$ where $K_1=\sigma(T_3)\cap \mathbb{T}$ and $K_2=\sigma(T_3) \cap r\mathbb{T}$. Since $T_3$ is normal, we have by Theorem \ref{Taylor} that there are closed subspaces $\mathcal{H}_{3}'$ and $\mathcal{H}_3''$ of $\HS_3$ reducing $T_3$ such that
	\[
		\mathcal{H}_{3}=\mathcal{H}_{3}' \oplus \mathcal{H}_{3}'', \quad \sigma(T|_{\mathcal{H}_3'})=K_1 \subseteq \mathbb{T} \quad \text{and} \quad \sigma(T|_{\mathcal{H}_3''})=K_2 \subseteq r\mathbb{T}.
	\]
	It shows that $T$ and $rT^{-1}$ are unitaries on $\mathcal{H}_3'$ and $\mathcal{H}_3''$ respectively. Thus, $\mathcal{H}_3'\subseteq \mathcal{H}_1$ and $\mathcal{H}_3'' \subseteq \mathcal{H}_2$ implying that 
		$\mathcal{H}_3=\mathcal{H}_3' \oplus \mathcal{H}_3'' \subseteq \mathcal{H}_1 \oplus \mathcal{H}_2$. Consequently, $\mathcal{H}_3=\{0\}$ and so, $\mathcal{H}=\mathcal{H}_1 \oplus \mathcal{H}_2$. Also, it is clear from the construction that $\HS_1$ and $\HS_2$ are the maximal closed reducing subspaces of $\HS$ on which $T$ and $rT^{-1}$ act as unitaries respectively. Hence, $\HS_1, \HS_2$ are uniquely determined and the proof is complete.
\end{proof}

Now we are in a position to give a proof to Theorem \ref{Main}, one of the main results of this article. \\

\noindent \textbf{Proof of Theorem \ref{Main}.} Let $T \in C_{1, r}$. Then $(T, rT^{-1})$ is a commuting pair of contractions acting on $\mathcal{H}$. It follows from Ando's dilation, Theorem \ref{Ando}, that there are commuting unitaries $U_1, rU_2^{-1}$ on a Hilbert space $\mathcal{K}_0\supseteq \mathcal{H}$ such that	 
\begin{equation} \label{eqn:001}
	 p(T, rT^{-1})h=P_\mathcal{H}p(U_1, rU_2^{-1})h
\end{equation}
for every $h \in \mathcal{H}$ and for every polynomial $p \in \C[z_1,z_2]$. Consequently, we have that 
\begin{equation} \label{eqn:002}
T^jh  =P_\mathcal{H}U_1^jh \quad \mbox{and} \quad \ T^{-j}h        =P_\mathcal{H}U_2^{-j}h,  \quad \mbox{for} \ h \in \mathcal{H} \ \text{and} \ j=0, 1, 2, \dotsc \,. 
\end{equation}
Let $f$ be a rational function with poles off $\overline{A}_r$. Then, $f(z)=p(z)q_1(z)^{-1}q_2(z)^{-1}$ as in (\ref{eq:1e}), where $q_1, q_2$ have their zeros in $\C \setminus \ov{\D}$ and $r\D$ respectively. For any $h \in \HS$, we have from Lemma $\ref{II_Laurent}$ that
\begin{equation*}
	\begin{split}
f(T)h&=p(T)\bigg(\overset{\infty}{\underset{n=0}{\sum}}q_{n1}T^n\bigg) \bigg(\overset{\infty}{\underset{n=0}{\sum}}q_{n2}T^{-n}\bigg)h\\
&=\lim_{m \to \infty} \bigg[ p(T)\bigg(\overset{m}{\underset{n=0}{\sum}}q_{n1}T^n\bigg) \bigg(\overset{m}{\underset{n=0}{\sum}}q_{n2}T^{-n}\bigg)\bigg]h \\
&=\lim_{m \to \infty} \bigg[ P_\mathcal{H}p(U_1)\bigg(\overset{m}{\underset{n=0}{\sum}}q_{n1}U_1^n\bigg) \bigg(\overset{m}{\underset{n=0}{\sum}}q_{n2}U_2^{-n}\bigg)\bigg]h  \quad [\text{by } (\ref{eqn:001}) \; \& \; (\ref{eqn:002})]\\
&=P_\mathcal{H}p(U_1)\bigg(\overset{\infty}{\underset{n=0}{\sum}}q_{n1}U_1^n\bigg) \bigg(\overset{\infty}{\underset{n=0}{\sum}}q_{n2}U_2^{-n}\bigg)h\\
&=P_\mathcal{H}p(U_1)q_1(U_1)^{-1}q_2(U_2)^{-1}h.\\
	\end{split}
\end{equation*}
Set
\[
 N=	\begin{bmatrix}
	         U_1 & 0\\
	         0 & U_2\\ 		
            \end{bmatrix} \ \mbox{and} \ \
      F=	\begin{bmatrix}
             0 & I_{\mathcal{K}_0}\\
             I_{\mathcal{K}_0} & 0\\ 		
            \end{bmatrix} \quad \text{ on } \; \mathcal{K}=\mathcal{K}_0 \oplus \mathcal{K}_0.
\]
Note that $F$ is a self-adjoint unitary. Since $U_1$ and $r^{-1}U_2$ are unitaries on $\mathcal{K}_0$, it follows from Theorem $\ref{Ar_unitary}$ that $N$ is an $A_r$-unitary. We have that $\HS \subseteq \mathcal K_0$. Let $V: \mathcal{H} \to \mathcal{K}$ be defined as $Vh=(h, 0)$. Evidently, $V$ is an isometric embedding and $V^*(x_1, x_2)=P_\mathcal{H}x_1$ for every $(x_1, x_2) \in \mathcal{K}$.
So, for any $h\in \HS$ we have that 
\begin{equation*}
	\begin{split}
		V^*\bigg(p(N)q_1(N)^{-1}Fq_2(N)^{-1}F\bigg)Vh&=
		V^*\bigg(p(N)q_1(N)^{-1}Fq_2(N)^{-1}\begin{bmatrix}
			0 & I_{\mathcal{K}_0}\\
			I_{\mathcal{K}_0} & 0\\ 		
		\end{bmatrix}\bigg) \begin{bmatrix}
		h \\
		0 \\ 		
	\end{bmatrix}\\
&=V^*\bigg(p(N)q_1(N)^{-1}F\begin{bmatrix}
	q_2(U_1)^{-1} & 0\\
	0 & q_2(U_2)^{-1}\\ 		
\end{bmatrix}\bigg) \begin{bmatrix}
	0 \\
	h \\ 		
\end{bmatrix}\\
&=V^*\bigg(p(N)q_1(N)^{-1}\begin{bmatrix}
	 0 & I_{\mathcal{K}_0}\\
	I_{\mathcal{K}_0} & 0\\ 		
\end{bmatrix}\bigg) \begin{bmatrix}
	0 \\
	q_2(U_2)^{-1}h \\ 		
\end{bmatrix}\\
&=V^*\begin{bmatrix}
	p(U_1)q_1(U_1)^{-1} & 0\\
	0 & p(U_2)q_1(U_2)^{-1}\\ 		
\end{bmatrix}  \begin{bmatrix}
	q_2(U_2)^{-1}h \\ 		
	0\\
\end{bmatrix}\\
&=V^*\begin{bmatrix}
	p(U_1)q_1(U_1)^{-1}q_2(U_2)^{-1}h \\ 		
	0\\
\end{bmatrix}\\
&=P_\mathcal{H}p(U_1)q_1(U_1)^{-1}q_2(U_2)^{-1}h\\
&=f(T)h.
	\end{split}
\end{equation*}
Also, we have
\[
FNF=\begin{bmatrix}
             0 & I_{\mathcal{K}_0}\\
             I_{\mathcal{K}_0} & 0\\ 		
            \end{bmatrix}
            \begin{bmatrix}
	         U_1 & 0\\
	         0 & U_2\\ 		
            \end{bmatrix}
            \begin{bmatrix}
             0 & I_{\mathcal{K}_0}\\
             I_{\mathcal{K}_0} & 0\\ 		
            \end{bmatrix}
            = \begin{bmatrix}
	         U_2 & 0\\
	         0 & U_1\\ 		
            \end{bmatrix}
\]
and hence $q_2(FNF)=Fq_2(N)F$. Note that $FNF$ is unitarily equivalent to $N$ and hence is an $\A$-unitary. Since $F$ is a self-adjoint unitary, we have that $q_2(FNF)^{-1}=Fq_2(N)^{-1}F$.
Putting everything together, we have that 
\[
f(T)h=V^*\bigg(p(N)q_1(N)^{-1}q_2(FNF)^{-1}\bigg)Vh
\]
for any rational function $f$ with poles off $\ov{A}_r$ and for every $h \in \mathcal{H}$. To see the converse,  assume that there is a Hilbert space $\mathcal{K} \supseteq \mathcal H$, an $A_r$-unitary $N$ on $\mathcal K$ and a self adjoint unitary $F $ on $\mathcal K$ such that 
 \[
 f(T)=P_\mathcal H \bigg(p(N)q_1(N)^{-1}q_2(FNF)^{-1}\bigg) \bigg|_\mathcal H
 \]
 for every rational function $f$ with poles off $\overline{A}_r$ $(\mbox{as in} \ (\ref{eq:1e}))$. Then $T=P_\mathcal{H} N|_\mathcal{H}$ and $T^{-1}=P_\mathcal{H}FN^{-1}F|_\mathcal{H}$. Since $N$ is an $A_r$-unitary, it follows from Theorem \ref{thm_3class} that $N \in C_{1, r}$. Thus $\|T\| \leq \|N\| \leq 1$ and $\|T^{-1}\| \leq \|FN^{-1}F\| \leq \|N^{-1}\| \leq r^{-1}$. Consequently, $T \in C_{1, r}$. The proof is now complete.
\qed

\begin{rem}
 
In Theorem \ref{Main}, if we denote the $\A$-unitary $FNF$ by $\widetilde{N}$, then it follows as a special case that for every $\A$-contraction $T \in \mathcal B(\HS)$, there is a Hilbert space $\mathcal K \supseteq \HS$ and an $\A$-unitary $\widetilde{N} \in \mathcal B(\mathcal K)$ such that
\[
f(T)=P_{\HS}f(\widetilde{N})|_{\HS}
\]
for every rational function of the form $f=1/q$ such that the zeros of $q$ lie inside $r\D$. Also, if $g=p/q_1$ with $q_1$ having its zeros inside $\C \setminus \ov{\D}$, then also we have
\[
g(T)= P_{\HS}g({N})|_{\HS},
\]
for some $\A$-unitary $N \in \mathcal B(\mathcal K_1)$ such that $\mathcal K_1 \supseteq \HS $.
\end{rem}

We conclude this Section by providing Theorem \ref{thm_C1r_pencil}, another main theorem of this paper that characterizes an operator in $C_{1, r}$ in different ways. Once again we mention that the equivalence of parts $(1)$ \& $(2)$ of this theorem follows from Proposition 3.2 in \cite{Pas-McCull}. However, we present here a different proof for this part also.\\

\noindent \textbf{Proof of Theorem \ref{thm_C1r_pencil}.} $(1) \implies (2)$. Let $T \in C_{1, r}$. It is easy to see that 
\begin{equation}\label{eqn_3.3}
D_T^2D_{(rT^{-1})^*}^2=D_{(rT^{-1})^*}^2D_T^2=(1+r^2)I_\HS-T^*T-r^2T^{-1}(T^{-1})^*.
\end{equation}
Consequently, $p(D_T^2)p(D_{(rT^{-1})^*}^2)=p(D_{(rT^{-1})^*}^2)p(D_T^2)$ for every polynomial $p\in \C[z]$. Choose a sequence of polynomials $p_n(x)$ that converges uniformly to $x^{1\slash 2}$ on the interval $0 \leq x \leq 1$. It follows from the spectral theorem that the sequence of operators $p_n(B)$ converges to $B^{1\slash 2}$ for any positive operator $B$ such that $0 \leq B \leq I_\HS$.  Applying (\ref{eqn_3.3}) to these polynomials and taking the limit as $n \to \infty$, we have
\[
D_TD_{(rT^{-1})^*}=D_{(rT^{-1})^*}D_T.
\]
Thus, for any $h \in \HS$, we have that
\[
\left\langle\left((1+r^2)I_\HS-T^*T-r^2T^{-1}(T^{-1})^*\right)h, h  \right\rangle =\left\langle D_T^2D_{(rT^{-1})^*}^2h, h  \right\rangle =\|D_TD_{(rT^{-1})^*}h\|^2 \geq 0.
\] 

	\medskip
	
\noindent $(2) \implies (3).$ Let $\Delta_T=(1+r^2)I_\HS-T^*T-r^2T^{-1}(T^{-1})^*\geq 0$ and let $P=(T^*T)^{1\slash 2}$. Note that $P$ is invertible as $T$ is invertible. Moreover, we have
\begin{equation*}
0 \leq P^*\Delta_T P=P((1+r^2)I_\HS-P^2-r^2P^{-2})P=(I_\HS-P^2)(P^2-r^2I_\HS)=\alpha(P^*, P).
\end{equation*}

\medskip

\noindent $(3) \implies (4).$ Let $P=(T^*T)^{1\slash 2} \in C_\alpha$. Then $\alpha(P^*, P)=(I_\HS-P^2)(P^2-r^2I_\HS) \geq 0$. Let $\lambda \in \sigma(P)$. It follows from the spectral theorem that $(1-\lambda^2)(\lambda^2-r^2) \geq 0$ and this holds if and only if $r \leq \lambda \leq 1$. Therefore, $\sigma(P) \subseteq \overline{A}_r$. Consequently, we have by Lemma \ref{lem_normal_Ar} that $P$ is an $\A$-contraction.

	\medskip
	
\noindent $(4) \implies (5).$ Let $P=(T^*T)^{1\slash 2}$ be an $\A$-contraction. For $U=TP^{-1}$, we have that
\[
U^*U=P^{-1}T^*TP^{-1}=P^{-1}P^2P^{-1}=I_\HS \quad \text{and} \quad UU^*=TP^{-2}T^*=T(T^*T)^{-1}T^*=I_\HS.
\]
Hence, $U$ is a unitary on $\HS$ and $T=UP$. 
	
	\medskip
	
\noindent $(5) \implies (1).$ Let $T=UP$ for a unitary $U$ and an $\A$-contraction $P$ on $\HS$. Then $rT^{-1}=(rP^{-1})U^*$. Consequently, we have that $\|T\| \leq \|P\| \leq 1$ and $\|rT^{-1}\| \leq \|rP^{-1}\| \leq 1$. The proof is now complete.
\qed

\vspace{0.2cm}
 
 \section{The quantum annulus}\label{Sec_quant}

\vspace{0.3cm}
 
\noindent In this Section, we provide a model for operators in the quantum annulus  $\mathbb{Q} \mathbb{A}_r$. Recall that 
\[
\mathbb{Q} \mathbb{A}_r=\{T: T \ \mbox{is invertible and} \ \|rT\|, \|rT^{-1}\| \leq 1\},
\] 
which is the quantization of the closed annulus $\overline{\mathbb{A}}_r=\{z \in \C: z \ne 0 \ \mbox{and} \ |rz|, |rz^{-1}|\leq 1\}$ in the sense that the scalars in the annulus are replaced by operators with similar norm-bounds. With this terminology, $C_{1, r}$ is nothing but the quantization of $\overline{A}_r$. It is evident that $C_{1,r}$ is a proper subset of  $\mathbb{Q} \mathbb{A}_r$. However, Lemma \ref{lem_equiv} shows that these two classes are actually equivalent. Thus, in light of Theorem \ref{Main}, we have an analogous model for $\mathbb Q \mathbb A_r$ in Theorem \ref{thm_404} whose proof is given below.\\

\noindent \textbf{Proof of Theorem \ref{thm_404}.}
The converse is straightforward. We assume that $T \in \mathbb{Q} \mathbb{A}_r$. We have by Lemma \ref{lem_equiv} that $rT \in C_{1, r^2}$. It follows from Theorem \ref{Main} that there is a Hilbert space $\mathcal{K} \supseteq \mathcal H$, an $A_{r^2}$-unitary $\tilde{N}$ on $\mathcal K$ and a self adjoint unitary $F $ on $\mathcal K$ such that 
 \begin{equation}\label{eqn_401}
f(rT)=P_\mathcal H \bigg(f_0(\tilde{N})f_1(\tilde{N})^{-1}f_2(F\tilde{N}F)^{-1}\bigg) \bigg|_\mathcal H
 \end{equation}
for every rational function $f=f_0\slash f_1f_2$ with zeros of $f_1$ and $f_2$ lying in $\C \setminus \overline{\D}$ and $r^2\D$ respectively. Let $g=p\slash q_1q_2$ with zeros of $q_1$ and $q_2$ in $\C \setminus r^{-1} \overline{\D}$ and $r\D$ respectively. We define
\[
f(z)=g(r^{-1}z)=\frac{p(r^{-1}z)}{q_1(r^{-1}z)q_2(r^{-1}z)},
\]
which is holomorphic on $\overline{A}_{r^2}$ with zeros of $q_1(r^{-1}z)$ and $q_2(r^{-1}z)$ lying in $\C \setminus \overline{\D}$ and $r^2\D$ respectively. We have by $(\ref{eqn_401})$ that 
 \[
 g(T)=f(rT)=P_\mathcal H \bigg(p(r^{-1}\tilde{N})q_1(r^{-1}\tilde{N})^{-1}q_2(r^{-1}F\tilde{N}F)^{-1}\bigg) \bigg|_\mathcal H.
 \]
Let $N=r^{-1} \tilde{N}$. Then $N$ is a normal operator and we have that
\[
\sigma(N)=\{r^{-1}\lambda : \lambda \in \sigma(\tilde{N})\}\subseteq \{r^{-1}\lambda : |\lambda|=1 \ \mbox{or} \ |\lambda|=r^2\}=\partial \overline{\mathbb{A}}_r.
\] 
The proof is now complete.
\qed 

Since we have an equivalence of the two classes $C_{1,r}$ and $\bA$ by Lemma \ref{lem_equiv}, we have a model in Theorem \ref{thm_C1r_MP} for $C_{1,r}$ analogous to the model for $\bA$ due to McCullough and Pascoe. We present a brief proof to this below.\\

\noindent \textbf{Proof of Theorem \ref{thm_C1r_MP}.}
The converse is trivial. Let us assume that $T \in C_{1, r}$. By Lemma \ref{lem_equiv}, $r^{-1\slash 2} T \in \mathbb{Q} \mathbb{A}_{\sqrt{r}}$. It follows from Theorem \ref{thm_model_QAr} that there is a unitary $U$ on some larger Hilbert space $\mathcal{K} \supseteq \HS$, an invertible operator $\tilde{J}$ on $\mathcal{K}$ such that 
$(r^{-1\slash 2}T)^n=P_\mathcal{H}\tilde{J}^n|_\mathcal{H}$ for all $n \in \mathbb{Z}$ and that
\[
\tilde{J}=U\begin{bmatrix}
r^{1\slash 2}I_{\mathcal{K}_{0}} & 0\\
0 & r^{-1\slash 2}I_{\mathcal{K}_{1}}
\end{bmatrix}\,,
\]
where $U$ is a unitary and $\mathcal{K}=\mathcal{K}_0 \oplus \mathcal{K}_1$. Take $J=r^{1\slash 2}\tilde{J}$ and the desired conclusion follows.
\qed 

\bigskip

Theorem \ref{thm_C1r_MP} gives a model for an operator in the $C_{1,r}$ class. Therefore, the model operator $J$ as in Theorem \ref{thm_C1r_MP} cannot be an $\A$-unitary. The reason is obvious; an $\A$-unitary can provide a model for the $\A$-contractions and the $C_{1,r}$ class is strictly bigger than that of the $\A$-contractions. So, a natural question arises: when the model operator $J$ becomes an $\A$-unitary so that the initial operator $T \in C_{1,r}$ becomes an $\A$-contraction ? Below we provide a necessary and sufficient condition for the same. 

\begin{prop}\label{prop_J_normal}
Let $J$ acting on a Hilbert space $\mathcal{K}$ be as in Theorem \ref{thm_C1r_MP}. Then the following are equivalent:
\begin{enumerate}
\item $J$ is an $\A$-unitary ;
\item $J$ is normal ;
\item $UJ_r=J_rU$ where $
J_r=\begin{bmatrix}
rI_{\mathcal{K}_0} & 0 \\
0 & I_{\mathcal{K}_1}
\end{bmatrix}$.
\end{enumerate}
\end{prop} 
 
\begin{proof}

$(1) \implies (2)$ follows trivially. We shall prove $(2) \implies (3) \implies (1)$.

\medskip 

\noindent $(2) \implies (3)$. Let $J=UJ_r$ be normal. Then $J^*J=JJ^*$ implies that $UJ_r^2=J_r^2U$. Since $J_r$ is self-adjoint, one can prove that $UJ_r=J_rU$ by an application of spectral theorem as discussed in the proof of Theorem \ref{thm_C1r_pencil}.

\medskip 

\noindent $(3) \implies (1)$. Let $UJ_r=J_rU$ and let $U=\begin{bmatrix}
U_{11} & U_{12} \\
U_{21} & U_{22} 
\end{bmatrix}$  with respect to $\mathcal{K}=\mathcal{K}_0 \oplus \mathcal{K}_1$. Then 
\[
0=UJ_r-J_rU=\begin{bmatrix}
rU_{11} & U_{12} \\
rU_{21} & U_{22} 
\end{bmatrix}-\begin{bmatrix}
rU_{11} & rU_{12} \\
U_{21} & U_{22} 
\end{bmatrix}=\begin{bmatrix}
0 & (1-r)U_{12} \\
(r-1)U_{21} & 0 
\end{bmatrix}\,,
\]
which is possible if and only if $U_{12}=U_{21}=0$ as $r<1$. Hence, $U_{11}$ and $U_{22}$ are unitaries on $\mathcal{K}_0$ and $\mathcal{K}_1$ respectively. Furthermore, we have that 
\[
J=\begin{bmatrix}
rU_{11} & 0 \\
0 & U_{22} 
\end{bmatrix}.
\]
It follows from Theorem \ref{Ar_unitary} that $J$ is an $\A$-unitary which completes the proof.
\end{proof}

\noindent \textbf{Proof of Theorem \ref{thm_QAr_pencil}.} Follows directly from Theorem \ref{thm_C1r_pencil} and Lemma \ref{lem_equiv}. 
\qed

\vspace{0.3cm}

\noindent \textbf{Concluding remark.} The model operators by Bello-Yakubovich \cite{Dmitry} for $C_{\alpha}$ class or by McCullough-Pascoe \cite{Pas-McCull} for $\mathbb Q \mathbb A_r$ cannot be normal or subnormal operators as a subnormal operator in $C_{1,r}$ is an $\A$-contraction by Proposition \ref{prop:new-021} and a subnormal operator in $\mathbb Q \mathbb A_r$ has $\ov{\mathbb A}_r$ as a spectral set by Proposition \ref{prop:new-022}. The chain $\mathcal A_r \subsetneq C_{\alpha} \subsetneq C_{1,r}$ clearly shows that an operator in $\mathcal A_r$ cannot be a model for $C_{1,r}$. Also Example \ref{ExampleII} confirms that the operators having $\ov{\mathbb A}_r$ as a spectral set is properly contained in $\mathbb Q \mathbb A_r$ and thus cannot be a model for the $\mathbb Q \mathbb A_r$. However, our models for both $C_{1,r}$ and $\mathbb Q \mathbb A_r$ consist of a pair of normal operators having their spectrums on the boundary of the annuli $\A$ and $\mathbb A_r$ respectively.


\vspace{0.2cm} 
 
 
 
\begin{thebibliography}{9}
		
		\vspace{0.2cm}
		
			
		\bibitem{Agler}
		J. Agler, \textit{Rational dilation on an annulus,} Ann. of Math., 121 (1985),  537 -- 563.\\
		
		\bibitem{A-H-R} J. Agler, J. Harland and B. J. Raphael, \textit{Classical function theory,
operator dilation theory, and machine computation on
multiply-connected domains}, Mem. Amer. Math. Soc.,  191 (2008), 289 -- 312.\\
		
		\bibitem{Ando}
		T. Ando, \textit{On a pair of commutative contractions}, Acta Sci. Math., 24 (1963), 88 -- 90.\\
		
		\bibitem{Dmitry}
		G. Bello and D. Yakubovich, \textit{An operator model in the annulus}, J. Operator Theory, 90 (2023), 25 -- 40. \\

	\bibitem{Conway}
		J. B. Conway, \textit{The theory of subnormal operators}, Amer. Math. Soc., Providence, Rhode Island, 1991.\\					
		
		\bibitem{DritschelII}
		M. A. Dritschel, M. T. Jury and S. McCullough, \textit{Dilations and constrained algebras}, Oper. Matrices, 10 (2016), 829 -- 861.\\
		

        \bibitem{DritschelI}
		M. Dritschel and B. Undrakh, \textit{Rational dilation problems associated with constrained algebras}, J. Math. Anal. Appl., 467 (2018), 95 -- 131.\\		
		
		\bibitem{McColloughIII}	 
		B. A. Mair and S. McCullough, \textit{Invariance of extreme harmonic functions on an annulus; applications to theta functions}, Houston J. Math., 20 (1994), 453 -- 473.\\

		\bibitem{McColloughIV}	
		S. McCullough and L. Shen, \textit{On the Szegő kernel of an annulus}, Proc. Amer. Math. Soc., 121 (1994), 1111 -- 1121.\\

		
		\bibitem{McColloughII}	
		S. McCullough, \textit{Matrix functions of positive real part on an annulus}, Houston J. Math., 21 (1995), 489 -- 506.\\
		
		
		\bibitem{comm lifting}
		S. McCullough and S. Sultanic, \textit{Agler-commutant lifting on an annulus}, Integral Equations Operator Theory, 72 (2012), 449 -- 482.\\		


		\bibitem{Pas-McCull}
		 S. McCullough and J. E. Pascoe, \textit{Geometric dilations and operator annuli}, J. Funct. Anal., 285 (2023), Paper No. 110035, 20 pp.\\		
		
		
		\bibitem{Misra}
		G. Misra, \textit{Curvature inequalities and extremal properties of bundle shifts}, J. Operator Theory, 11 (1984), 305 -- 317. \\
		
\bibitem{Mittal}
M. Mittal, \textit{Function theory on the quantum annulus and other domains}, Thesis (Ph.D.)-University of Houston, ISBN: 978-1124-46385-8, 2010, 141 pp.\\		
		
		\bibitem{Taylor}
		J. L. Taylor, \textit{The analytic-functional calculus for several commuting operators}, Acta Math., 125 (1970), 1 -- 38.\\

		
		\bibitem{Tsikalas}
		G. Tsikalas, \textit{A von Neumann type inequality for an annulus}, J. Math. Anal. Appl., 506 (2022), Paper No. 125714, 12 pp.\\
		
		\bibitem{TsikalasII}
		G. Tsikalas, \textit{A note on a spectral constant associated with an annulus}, Oper. Matrices, 16 (2022), 95 -- 99.\\		
		
		\bibitem{Williams}
		J. P. Williams, \textit{Minimal spectral sets of compact operators}, Acta Sci. Math., 28 (1967), 93 -- 106.\\		
			
		
	\end{thebibliography}
\end{document}